\documentclass[12pt,a4paper]{amsart}

\title{Affine Calculus for Constrained Minima of the Kullback-Leibler Divergence}
\author{Giovanni Pistone \href{https://orcid.org/0000-0003-2841-788X}{orcid}}
\address{De Castro Statistics, Collegio Carlo Alberto}
\email{giovanni.pistone@carloalberto.org}
\address{Nuovo SEFIR, c/o Coworld, Centro Direzionale Milano Due, Palazzo Canova, I-20054 Segrate}

\usepackage{accents}
\usepackage{fullpage}
\usepackage{hyperref}
\usepackage{mathtools}
\usepackage{tikz-cd}
\usepackage{xcolor}
\usepackage{cleveref} 

\DeclareMathOperator{\Cov}{Cov}
\DeclareMathOperator{\Entropy}{\mathcal H}
\DeclareMathOperator{\Expectation}{\mathbb E}
\DeclareMathOperator{\Grad}{grad}
\DeclareMathOperator{\Hessian}{Hess}
\DeclareMathOperator{\Maxexp}{\mathcal E}
\DeclareMathOperator{\Var}{Var}
\newcommand{\JS}[2]{\operatorname{JS}\left(#1,#2\right)}
\newcommand{\KL}[2]{\operatorname{D}\left(#1\,\middle\Vert#2\right)}
\newcommand{\condexpat}[3]{\Expectation_{#1}\left(#2 \middle| #3\right)}
\newcommand{\covat}[3]{\Cov_{#1}\left(#2,#3\right)}
\newcommand{\derivby}[1]{\frac{d}{d#1}}
\newcommand{\entropyof}[1]{\Entropy\left(#1\right)}
\newcommand{\etransport}[2]{\prescript{\text{e}}{} {\mathbb U} _ {#1} ^ {#2}}
\newcommand{\euler}{\mathrm{e}}
\newcommand{\expbundleat}[1]{S\maxexpat{#1}}
\newcommand{\expectat}[2]{\Expectation_{#1}\left[#2\right]}
\newcommand{\expfiberat}[2]{S_{#1}\maxexpat{#2}}

\newcommand{\logof}[1]{\log\left(#1\right)}
\newcommand{\maxexpat}[1]{\Maxexp\left(#1\right)}
\newcommand{\mixbundleat}[1]{\prescript{*}{}S\maxexpat{#1}}
\newcommand{\mixfiberat}[2]{\prescript{*}{}S_{#1}\maxexpat{#2}}
\newcommand{\mtransport}[2]{\prescript{\text{m}}{} {\mathbb U} _ {#1} ^ {#2}}

\newcommand{\reals}{\mathbb{R}}
\newcommand{\scalarat}[3]{\left\langle#2,#3\right\rangle_{#1}}
\newcommand{\setof}[2]{\left\{#1 \, \middle| \, #2 \right\}}
\newcommand{\varat}[2]{\Var_{#1}\left(#2\right)}
\newcommand{\velocity}[1]{\accentset{\star}{#1}}

\newtheorem{Proposition}{Proposition}

\begin{document}

\begin{abstract}
The non-parametric version of  Amari's dually affine Information Geometry provides a practical calculus to perform computations of interest in statistical machine learning. The method uses the notion of a statistical bundle, a mathematical structure that includes both probability densities and random variables to capture the spirit of Fisherian statistics. We focus on computations involving a constrained minimization of the Kullback-Leibler divergence. We show how to obtain neat and principled versions of known computation in applications such as mean-field approximation, adversarial generative models, and variational Bayes.
\end{abstract}

\maketitle

\tableofcontents

\section{Introduction and Notations}
\label{sec:short-recap}

{ Many modern Artificial Intelligence (AI) and machine learning (ML) algorithms are based on non-parametric statistical methods and optimization algorithms based on the minimization of a divergence measure between probability functions. In particular, one computes the gradient of a function defined on the probability simplex; then, the learning uses a gradient ascent technique. Such a basic approach is illustrated, for example, in the textbook \cite{efron|hastie:2016} (Ch.~18).

In most papers, ordinary convex calculus tools on the open probability simplex provide the relevant derivatives and gradients. The relation between the analytic computations and their statistical meaning is not exposed. This paper focuses on the derivative and gradient computations by providing the geometric framework called Information Geometry (IG). This geometry differs from the usual convex analysis because its devices have a direct statistical meaning. For example, the velocity of a one-dimensional parametric curve \mbox{$\theta \mapsto p(\theta)$} in the open probability simplex is defined to be the Fisher's score $\derivby \theta \log p(\theta)$ instead of the ordinary derivative $\derivby \theta p(\theta)$. Generally speaking, IG is a geometric interpretation of Fisherian inference (\cite{efron|hastie:2016}, Ch.~5).

Amari's Information Geometry (IG) \cite{amari:97isi,amari:1998natural,amari|nagaoka:2000} has been successfully applied to modern AI algorithms; see, for example, \cite{amari:2016}. Here, we use the non-parametric version of IG of \cite{pistone|sempi:95,chirco|pistone:2022}. This version is \emph{{non-parametric}} 
because the basic set of states is the open probability simplex, it is \emph{{affine}}, as it satisfies a generalization of the classical Weyl's axioms \cite{weyl:1952}. Moreover, it is \emph{{dually affine}} in the sense already defined in Amari contributions because the covariance bilinear operator appears as a duality pairing in the vector space of coordinates.

The specific applications we will consider as examples come from the literature in statistical ML, particularly those that involve the constrained minimization of the Kullback--Leibler divergence (KL-divergence). Indeed, our main result in \cref{sec:general} is a form of the total gradient of the KL-divergence as expressed in the dually affine geometry. Namely, we consider symmetric divergences \cite{lin:1991divergence}, generative adversarial networks \cite{goodfellow|pouget-abadie|mirza|xu|warde-farley|ozair|courville|bengio:2014}, mixed entropy and transport optimization \cite{amari|karakida|oizumi:2018INGE,peyre|cuturi:2019}, and variational Bayes \cite{kingma|welling:2022autoencodingvariationalbayes,khan|rue:2023}.

Non-parametric IG stands in general sample spaces and under various functional assumptions. One option, among many, is the use of Orlicz spaces \cite{musielak:1983}; see \cite{pistone|sempi:95,pistone:2018-igaia-iv}. In this paper, we are not interested in discussing the functional setup. Still, we are interested in presenting the peculiar affine calculus of positive probability functions on a finite state space $\Omega$ in a geometric language compatible with the infinite-dimensional theory \cite{lang:1995}. Such a calculus provides principled definitions of a curve's \emph{{velocity}}, the scalar field's \emph{{gradient}}, and the \emph{{gradient flow}}. 

\subsection{Prerequisites}

Below, we provide a schematic summary of the theory. For complete details, we refer to previous presentations in \cite{pistone:2020-NPCS,chirco|pistone:2022}.

Let $\Omega$ be a finite sample space.} We look at the open simplex as the \emph{{maximal exponential model}} denoted as $\maxexpat \Omega$. {In fact, we present every couple of positive probability functions on $\Omega$, say $p,q$, in the form inspired by Statistical Physics \cite{landau|lifshits:1980}:
  \begin{equation}
    \label{eq:maxexp}
    q = \euler^{v - \Psi} \cdot p \ , 
  \end{equation}
  where $p$ represents a ground state, $q$ is a perturbation of the ground state, $v$ is a random variable, $\Psi$ is a normalizing constant, and $\psi = \kappa_p(v) = \log \expectat p {\euler^v}$ is the cumulant function.
  
  The random variable $v$ depends on $p$ and $q$ up to a constant. If we specify $\expectat p v = 0$ in \cref{eq:maxexp}, then a straightforward computation gives
  \begin{gather*}
    \label{eq:K}
    v = \log \frac q p - \expectat p {\log \frac q p} \ , \quad q = \euler^{v - \kappa_p(v)} \cdot p \ , \\
    \label{eq:KL} \kappa_p(v) = \log \expectat p {\euler^v} = \expectat p {\log \frac p q} = \KL p q \ ,
  \end{gather*}
  where {$\KL p q$} is the KL-divergence. Regarding the entropy,
  \begin{equation*}
    \label{eq:entropy}
    \KL p q = \expectat p {\log \frac p q} = \expectat p {\log p} - \expectat p {\log q} = - \entropyof p + \entropyof{p,q} \ . 
  \end{equation*}

  If we specify $\expectat q v = 0$ in \cref{eq:maxexp}, an analogous computation gives
  \begin{gather*}
    \label{eq:K-}
    v = \log \frac q p - \expectat q {\log \frac q p} \ , \quad q = \euler^{v - \kappa_p(v)} \cdot p \ , \\
    \label{eq:KL-} \kappa_p(v) = \log \expectat p {\euler^v} = - \expectat q {\log \frac p q} = - \KL q p \ .
  \end{gather*}

A vector bundle is a collection of vector spaces, and each vector space is called a \emph{{fiber}} of the bundle. For example, the tangent bundle collects all tangent vectors at each point in differential geometry. In Fisher's statistics of the open probability simplex, one considers the vector space of all Fisher's scores of one-dimensional models through the probability function $q$. Inspired by this last example, we call the \emph{{statistical bundle}} the vector bundle $\expbundleat \Omega$ of all couples $(q,v)$ of a positive probability function $q$ and a $q$-centered random variable, $\expectat q v = 0$,
\begin{equation*}
    \expbundleat \Omega = \setof{(q,v)}{q \in \maxexpat{\Omega}, \expectat q v = 0} \ .
\end{equation*}

{Each} fiber $\expfiberat q \Omega$ is a Euclidean space for the covariance inner product $\scalarat q v w = \expectat q {vw}$. 

The covariance inner product is both a Riemannian metric and a duality pairing. The metric interpretation leads to the Riemannian version of IG. The duality pairing interpretation leads to our dually affine IG. Because of that, we want to distinguish between the fibers $\expfiberat p \Omega$ and the dual fibers $\mixfiberat p \Omega$. The first bundle is called \emph{{exponential bundle}}, while the second bundle is called \emph{{mixture bundle}}. We use the notation
  \begin{equation*}
   \mixfiberat p \Omega \times \expfiberat p \Omega \ni (v,w) \mapsto \scalarat p v w \ , \quad p \in \maxexpat \Omega \ .  
  \end{equation*}
  
{In} our setup, all the vector spaces of random variables are finite-dimensional; hence, the fibers $\expfiberat p \Omega$ and $\mixfiberat p \Omega$ are equal vector spaces. However, it is a useful distinction, as it will be apparent in the discussion of parallel transports below.} 

The definition of the statistical bundle aims to capture an essential mechanism of Fisher's approach to statistics (\cite{efron|hastie:2016}, Ch.~4). Suppose $t \mapsto q(t) \in \maxexpat \Omega$ is a one-dimensional statistical model. In that case, the Fisher's score is $t \mapsto \derivby t \log q(t) = \velocity q(t)$, and $t \mapsto (q(t),\velocity q(t)) \in \expbundleat \Omega$ { is the \emph{{lift}} of the curve to the statistical bundle}.  

Dually affine geometry follows from the definition of two parallel transports on the fibers and two affine charts. The parallel transports act between the fibers{
\begin{gather} 
\expfiberat q \Omega \ni v \mapsto \etransport q r v = v - \expectat r v \in \expfiberat r \Omega \quad \textit{{is the exponential transport,}} \label{eq:etransport} \\
 \mixfiberat q \Omega \ni w \mapsto \mtransport q r w = \frac q r v \in \mixfiberat r \Omega \quad \textit{is the mixture transport.} \label{eq:mtransport}
\end{gather}

It is easy to check that the transports are duals of each other:
\begin{gather}
  \scalarat q v {\etransport r q w} = \scalarat r {\mtransport q r v} {w} \quad \textit{is the transport's duality,}\label{eq:transport-duality} \\
  \scalarat q v w = \scalarat r {\mtransport q r v}{\etransport q r w} \quad \textit{is the inner product push.} \label{eq:inner-prod-push}
\end{gather}

The affine charts that define the two dual affine geometries by mapping the base set to a vector space of coordinates are
\begin{gather}
  q \mapsto s_p(q) = \log \frac q p - \expectat p {\log \frac q p} \in \expfiberat p \Omega \quad \textit{is the exponential chart} \label{eq:exp-chart} \\
  q \mapsto \eta_p(q) = \frac q p - 1 \in \mixfiberat p \Omega \quad
  \textit{is the mixture chart,} \label{eq:mix-chart}
\end{gather}}and the geometries defined by the two atlases are affine because the parallelogram law holds in both cases:
\begin{gather*}
  s_p(q) + \etransport q p s_q(r) = s_p(r) \label{eq:chasles-exp} \\
  \eta_p(q) + \mtransport q p \eta_q(r) = \eta_p(r) \label{eq:chasles-mix}
\end{gather*}

The inverse of the exponential chart is a non-parametric exponential family (\cite{efron|hastie:2016}, Ch.~5), and the known mechanisms of the cumulant function provide a fundamental calculus tool~\cite{brown:86}. {If $K_p$ is the restriction of $\kappa_p$ to $\expfiberat p \Omega$ and $v = s_p(q)$ then $q = s_p^{-1}(v) = e_p(v) = \euler^{v - K_p(v)}\cdot p$,
\begin{gather}
K_p(v) = \log \expectat p {\euler^v} \quad \textit{cumulant function} \label{eq:Kp} \\
  \KL p {e_p(v)} = K_p(v) \quad
  \textit{cumulant function express the KL- divergence,} \label{eq:Kp-equals-Dpq}
  \\   
  dK_p(v)[h] = \expectat {e_p(v)} h \quad \textit{is the derivative of $K_p$ in the direction $h$,}  \label{eq:Kp-deriv} \\
  d^2K(v)[h,k] = \covat {e_p(v)} h k \quad \textit{is the second derivative of $K_p$ in the directions $h,k$.} \label{eq:Kp-2-deriv}
\end{gather}
\cref{eq:Kp-deriv,eq:Kp-2-deriv} are the non-parametric version of the well-known properties of the derivative of the cumulant function in exponential models; see (\cite{efron|hastie:2016}, \S~5.5) and \cite{brown:86}.}

We can now show that Fisher's score is a velocity in the technical sense of a velocity computed in the moving frame of both charts. If $t \mapsto q(t) \in \maxexpat \Omega$ is a smooth curve, and $\Phi \colon \maxexpat \Omega \to \reals$ is a smooth mapping,
\begin{gather}
  \velocity q(t) = \left. \derivby t s_p(q(t)) \right|_{p=q(t)} = \left. \derivby t \eta_p(q(t)) \right|_{p=q(t)} = \derivby t \log q(t) \quad \textit{is the velocity,} \label{eq:velocity} \\
  \derivby t \Phi(q(t)) = \scalarat {q(t)} {\Grad \Phi(q(t))} {\velocity q(t)} \quad \textit{is the gradient.} \label{eq:gradient}
\end{gather}

{{The} squared norm of the velocity \eqref{eq:velocity},
\begin{equation*}
    \scalarat{q(t)}{\velocity q(t)}{\velocity q(t)} = \expectat{q(t)}{\left(\derivby t \log q(t)\right)^2} = \sum \frac{{\dot q(t)}^2} {q(t)} \ ,
\end{equation*}
is the Fisher information that appeared first in the classical 
Cramer--Rao lower bound.}

The gradient defined in \cref{eq:gradient} is frequently called the \emph{{natural gradient}} in the IG literature, following the use introduced in the case of parametric models by Amari \cite{amari:1998natural}. {In Riemannian geometry \cite{docarmo:1992,lang:1995}, the metric acts as a duality pairing, and the definition of the gradient is similar to \cref{eq:gradient}. The classic example of the computation of the gradient is the gradient of the expected value as a function of the probability function,
\begin{equation*}
\derivby t \expectat {q(t)} u = \derivby t \sum u \frac {\dot q(t)} {q(t)} \, q(t) = \scalarat {q(t)} {u - \expectat {q(t)} u}{\velocity q(t)} \ ,
\end{equation*}
so that $\Grad \expectat q u = u - \expectat q u$.

The gradient of $\Phi$ gives the velocities of curves ``orthogonal'' to the surfaces of constant $\Phi$-value, that is, the curves of steepest ascent. The solutions of the equation $\Grad \Phi = 0$ are the stationary points of $\Phi$, and an equation of the form 
\begin{equation*}
    \velocity q(t) = \epsilon(t) \Grad \Phi(q(t))
\end{equation*}
is a \emph{{gradient flow equation}}.

In conclusion, we review the derivation of a function $f$ between two maximal exponential models using the mixture charts \cref{eq:mix-chart}. The expressions of $f$ and its derivative $df$ in the charts centered, respectively, at $p_1$ and $p_2$, are
\begin{equation*}
\begin{array}{lcr}
\begin{tikzcd}
      \maxexpat {\mu_1} \arrow[r,"f"] & \maxexpat {\mu_2} \arrow[d, "\eta_{p_2}"] \\
        \expfiberat {p_1} {\mu_1} \arrow[r,"f_{p_1 \cdot p_2}"'] \arrow[u,"n_{p_1}^{-1}" ] & \expfiberat {p_2} {\mu_2} 
\end{tikzcd}
&\text{and}&
\begin{tikzcd}[column sep=huge]
     \expfiberat {q_1} {\mu_1} \arrow[r,"df(q_1)"] & \expfiberat {f(q_1)} {\mu_2} \arrow[d, "\mtransport {f(p_1)} {q_2}"] \\
       \expfiberat {p_1} {\mu_1} \arrow[r,"df_{p_1 \cdot p_2}(\eta_{p_1}(q_1)"'] \arrow[u, "\mtransport {p_1} {q_1}"] & \expfiberat{p_2}{\mu_2}
\end{tikzcd}
\end{array}
\end{equation*}

It follows that the computation of the derivative from its expression is
\begin{equation} \label{eq:bundle-derivative} 
  df(q)[\velocity q] = \mtransport {p_2}{q} df_{p_1,p_2}(\eta_{p_1}(q))[\mtransport {q} {p_1} \velocity q] \ .
\end{equation}
}

{
\subsection{Summary of Content}

 In the following sections, we give both new results and new versions of the known results. The aim is to show the interest of the non-parametric dually affine IG in computing the gradient flow of a constrained KL-divergence. }
 
In \cref{sec:general}, we show how to use the statistical bundle formalism to compute derivatives of functions defined on the open probability simplex and how to compute natural gradients and total natural gradients of the KL-divergence, the cross entropy, the entropy, and the Jensen--Shannon divergence.

{In \cref{sec:product-sample-space}, we apply the general computations of the previous section to independence models and marginal conditional probabilities in a factorial product setting. The dually affine methodology methodically reproduces known computations and suggests neat variations of potential interest. {In particular, \cref{sec:variational-bayes} contains a fully worked example of the derivation of a gradient flow equation of interest in approximate Bayes~computations.}}

\section{Total Natural Gradient of the KL-Divergence}
\label{sec:general}

The KL-divergence (\cite{amari:2016}, Ch.~3) as a function of two variables is
\begin{equation} \label{eq:KLx2}
  D \colon   \maxexpat\Omega \times \maxexpat \Omega \ni (q,r) \mapsto \KL q r = \expectat q {\log \frac q r} \ . 
\end{equation}

{The} computation of the total derivative is well-known in Information Theory. However, we provide proof in the affine setting, expressing the result in the affine charts.

In the exponential chart at $p$ and in the mixture chart at $p$, the expressions of the probability functions $q$ and $r$ are, respectively,
\begin{equation} \label{eq:q-and-r-at-p}
    q = e_p(v) = \euler^{v - K_p(v)} \cdot p \ , \quad r = \eta_p^{-1}(w) = (1+w) \cdot p \ ,
  \end{equation}
  
{By plugging \eqref{eq:q-and-r-at-p} into \eqref{eq:KLx2} and using \cref{eq:Kp-deriv},} one sees that the expressions of the partial KL-divergences are, respectively,
{
  \begin{multline} 
  \KL {e_p(v)} r = \expectat {e_p(v)} {v  - K_p(v)} - \expectat {e_p(v)} {\log \frac r p} = \\
  \expectat {e_p(v)} {v  - K_p(v)} - \expectat {e_p(v)} {s_p(v)} - \KL p r = \\
  dK_p(v)[v] - K_p(v) - dK_p(v)[s_p(r)] - \KL p r \ , \label{eq:KL1-at-p}
\end{multline}
}and
\begin{equation}
  {\KL q {\eta_p^{-1}(w)} = \KL q p - \expectat q {\log(1+w)}} \label{eq:KL2-at-p} \ .
\end{equation}

{Notice that the peculiar choice of the charts in the combination exponential for the first variable and mixture for the second variable is inessential in the finite state space case because any other choice will produce the same final result in the computation of the total natural gradient. However, it is consistent with the dual affine setting, in which two connections exist between one space and its dual. However, the expression of the KL-divergence using the exponential chart in both variables is interesting because, in such a case, the resulting expression is equal to the Bregman divergence of the cumulant function~$K_p$,
\begin{equation*}
  \KL {e_p(v)}{e_p(w)} = K_p(w) - K_p(v) - dK_p(v)[w-v] \ ,
\end{equation*}
which, in turn, is the second remainder in the Taylor expansion. For example, one closed form is
\begin{equation*}
   \KL {e_p(v)}{e_p(w)} = \int_0^1\int_0^1 d^2K(v + st (w -v))[w-v,w-v] \, sds\,dt \ . 
  \end{equation*}

  If $q(s,t) = e_p(v + st (w -v)) \propto q^{1-st}r^{st}$, then by \cref{eq:Kp-2-deriv},
  \begin{equation*}
    \KL q r = \iint_0^1 \varat {q(s,t)} {\log \frac r q} \, s \, ds \, dt \ . 
  \end{equation*}
} 

\subsection{Total Natural Gradient of the KL-Divergence}
\label{sec:total-natur-grad}

{We compute our gradients in the duality induced on each fiber by the covariance; hence, the} \emph{{total natural gradient}} of the KL-divergence has two components implicitly defined~by
\begin{equation} \label{eq:total-nat-grad-KL}
    \derivby t \KL {q(t)} {r(t)} = \\ \scalarat {q(t)} {\velocity q(t)} {\Grad_1 \KL{q(t)}{r(t)}} + \scalarat {r(t)} {\Grad_2 \KL{q(t)}{r(t)}}{\velocity r(t)} \ ,
\end{equation}
{where $\Grad_1 \KL{q}{r}$ is a random variable in the fiber at $Q$, while $\Grad_2 \KL{q}{r}$ is a random variable in the fiber at $r$. The adjective total refers to the fact that $D$ is a function of two~variables.}
\begin{Proposition} \label{prop:KL-tot-nat-grad}
The total natural gradient of the KL-divergence is
  \begin{equation}
    \label{eq:KL-total-grad}
    (q,r) \mapsto \Grad \KL q r = \left(- s_q(r), -\eta_r(q)\right) \in \expbundleat \Omega \times \mixbundleat \Omega \ .
  \end{equation}
  
{That is, more explicitly, for each smooth couple of curves $t \mapsto q(t)$ and $t \mapsto r(t)$, {~\cref{eq:total-nat-grad-KL}~becomes}}
  \begin{equation*}
    \label{eq:KL-total-grad-dy}
    \derivby t \KL {q(t)} {r(t)} = - \scalarat {q(t)}  {\velocity q(t)} {s_{q(t)}(r(t))}  - \scalarat {r(t)} {\eta_{r(t)}(q(t))} {\velocity r(t)} \ .
    \end{equation*}    
\end{Proposition}

\begin{proof}
{From \cref{eq:Kp-2-deriv}, the derivative at $v \in \expfiberat p \Omega$ of \cref{eq:KL1-at-p} in the direction $h = \etransport q p \velocity q$ is}
\begin{multline*}  \label{eq:deriv-KL1-at-p}
  d^2K_p(v)[v,h] - dK_p(v)[h] + dK_p(v)[h] - d^2K_p(v)[s_p(r),h] = \\
  \covat q {s_p(q) - s_p(r)} {h} = \covat q {\log \frac q r - \expectat p {\log \frac q r}} {h} = \\
  \expectat q {\left(\log \frac q r - \expectat q {\log \frac qr}\right)\left(h - \expectat q h\right)} = \scalarat q {-s_q(r)} {\velocity q} \ .
\end{multline*}

The derivative at $w \in \mixfiberat p \Omega$ {of \cref{eq:KL2-at-p}  in the direction $k = \mtransport r p \velocity r$ is}
\begin{equation*}  \label{eq:deriv-KL2-at-p}
- \expectat q {\frac p r k} = - \expectat r {\frac q r \velocity r} = \scalarat r {-\eta_r(q)}{\velocity r} \ .
\end{equation*}
\end{proof}

{The gradient computation forms the corresponding gradient flow equation, whose discretization provides basic optimization algorithms. Here are two basic examples.}

Given $r \in \maxexpat \Omega$, the solution of the gradient flow equation
\begin{equation*}
  \velocity q(t) = - \Grad_1 \KL {q(t)} r = s_q(r) \ , \quad q(0) = q_0 \ ,
\end{equation*}
is the exponential family
\begin{equation}
  \label{eq:exponential}
  t \mapsto q(t) = \euler^{\euler^{-t}v_0 - K_r(\euler^{-t}v_0)} \cdot r \ , \quad v_0 = s_r(q_0) \ .
\end{equation}

{In} fact, the LHS of \cref{eq:exponential} is given by
\begin{gather*}
  \label{eq:LHS-exponential}
  \log q(t) = \euler^{-t}v_0 - K_r(\euler^{-t}v_0) + \log r \ , \\
  \velocity q(t) = - \euler^{-t} v_0 + dK_r(\euler^{-t}v_0)[\euler^{-t} v_0] \ ,
\end{gather*}
while the RHS is 
\begin{multline*}s_{q(t)}(r) = - \euler^{-t}v_0 + K_r(\euler^{-t}v_0) - \expectat {q(t)} {- \euler^{-t}v_0 + K_r(\euler^{-t}v_0)} = \\
  - \euler^{-t}v_0 + \expectat {q(t)} {\euler^{-t}v_0} \ .
\end{multline*}

{The} conclusion follows from \cref{eq:Kp-deriv}.

Given $q \in \maxexpat \Omega$, the solution of the gradient flow equation
\begin{equation*}
  \velocity r(t) = - \Grad_2 \KL {q} {r(t)} = \eta_{r(t)}(q) \ , \quad r(0) = r_0 \ ,
\end{equation*}
is the mixture  family
\begin{equation}
  \label{eq:mixture}
  t \mapsto r(t) = \euler^{-t}r_0 + (1 - \euler^{-t}) q \ .
\end{equation}

{The} LHS of \cref{eq:mixture} is
\begin{equation*}
  \label{eq:LHS-mixture}
  \velocity r(t) = \frac {\dot r(t)} {r(t)} = \frac {- \euler^{-t} r_0 + \euler^{-t} q}{\euler^{-t}r_0 + (1 - \euler^{-t}) q} = \frac {q - r_0}{r_0 + (\euler^{t}-1) q} \ ,
\end{equation*}
while the RHS is
\begin{equation*}
  \label{eq:RHS-mixture}
  \eta_{r(t)}(q) = \frac q {\euler^{-t}r_0 + (1 - \euler^{-t}) q} - 1 = \frac
 {q - \euler^{-t}r_0 + (1 - \euler^{-t}) q} {\euler^{-t}r_0 + (1 - \euler^{-t}) q} = \frac {q - r_0}{r_0 + (\euler^{t}-1) q} \ . 
\end{equation*}

{ Notice that in both cases, the $t$ parameter appears in the solution in exponential form. Other forms of the temperature parameter will follow from a weighted form of the gradient flow equation.} 

\subsection{Natural Gradient of the Entropy and Total Natural Gradient of the Cross Entropy}

The KL-divergence equals the cross entropy minus the entropy,  
  \begin{equation*} \label{eq:kl-equals-cross-ent-minus-ent}
    \KL q r =  \expectat q {-\log r} - \expectat q {-\log q}= \entropyof {q,r} - \entropyof q \ . 
  \end{equation*}

In the exponential chart at $p$ for the first variable, the cross entropy is
  \begin{multline}
    \label{eq:cros-entropy-express-1}
    \entropyof{s^{-1}_p(v),r} = \\ \expectat {e_p(v)}{- \log r} = \expectat {e_p(v)} {- \log r - \expectat p {- \log r}} + \entropyof{p,r} = \\ dK_p(v)[- \log r - \expectat p {- \log r}] + \entropyof{p,r} \ ,
  \end{multline}
  with derivative at $v$ in the direction $h$
 \begin{multline*}
    \label{eq:cros-entropy-grad-1}
    d^2K_p(v)[- \log r - \expectat p {- \log r},h] = \covat q {- \log r} {h} = \\ \expectat q {\left(- \log r - \entropyof{q,r}\right) \left(h - \expectat q h\right)} = \scalarat q {\velocity q} {- \log r - \entropyof{q,r}} \ .
  \end{multline*}
 
  In the mixture chart at $p$ for the second variable
\begin{equation}
  \label{eq:cros-entropy-express-2}
  \entropyof{q,\eta^{-1}_p(w)} = \expectat q {- \logof{(1+w) \cdot p}} = \expectat q {- \logof{1+w}} + \entropyof{q,p} \ ,  
\end{equation}
with derivative at $w$ in the direction $k$,
\begin{equation*}
  \label{eq:cros-entropy-grad-2}
\expectat q {-(1+w)^{-1} k} = \expectat r {- \frac q r \mtransport p r k} = \expectat r {- \left( \frac q r -1\right) \velocity r} = \scalarat r {\eta_r(q)}{\velocity r} \ .   
\end{equation*}

\begin{Proposition}
The total natural gradient of the cross entropy is
\begin{equation*}
  \label{eq:cros-entropy-grad}
  \Grad \entropyof{q,r} = (-\log r - \entropyof{q,r}, \eta_r(q))
\end{equation*}
and the natural gradient of the entropy is  
\begin{equation}
  \label{eq:entropy-grad}
  \Grad \entropyof q = - \log q - \entropyof q \ .
\end{equation}
\end{Proposition}

\begin{proof}
  The first statement follows from Equations \eqref{eq:cros-entropy-express-1} and \eqref{eq:cros-entropy-express-2}.
From the decomposition $\entropyof q = \entropyof{q,r} - \KL q r$, we find the gradient of the entropy,
\begin{equation*}
  \Grad \entropyof q = - \log r - \entropyof{q,r} + s_q(r) = - \log q - \entropyof q \ .
\end{equation*}
\end{proof}

\subsection{Total Natural Gradient of the Jensen--Shannon Divergence}
\label{sec:jens-shan-diverg}
The \emph{{Jensen--Shannon}} divergence \cite{lin:1991divergence} is
\begin{multline}
  \label{eq:JSD}
  \JS q r = \frac12 \KL q {\frac12(q+r)} + \frac12 \KL r {\frac12(q+r)}  = \\
  \entropyof{\frac12(q+r)} - \frac12 \entropyof q - \frac12 \entropyof r \ .
\end{multline}

{It} is the minimum value of the function
\begin{equation*}
    \phi \colon p \mapsto \frac12 \left(\KL q p + \KL r p \right)
 \ .
 \end{equation*}

{In} fact,
\begin{equation*}
 \Grad \phi(p) = - \frac12 \left(\eta_p(q) + \eta_p(r) \right) = \frac12 \left(-\frac q p + 1 - \frac r p + 1 \right)= - \frac {\frac12(q+r)}{p} + 1   
\end{equation*}
which vanishes for $p = \frac12(q+r)$.

Let us compute the derivative of $f \colon q \mapsto \frac12(q+r)$. The mixture expression of $f$ at $p$ according to \cref{eq:mix-chart} is the affine function
\begin{equation*}
  \label{eq:translation-expression}
  f_p(v) = \eta_p \circ f \circ \eta^{-1}_p(v) = \frac{\frac12(1+v) \cdot p + r}p - 1 = \frac12 v + \frac12 \eta_p(r) \ ,
\end{equation*}
so that the derivative in the direction $h$ is $df_p(v)[h] = h/2$.

The push-back, according to the mixture transport \cref{eq:mtransport}, is
\begin{equation}
  \label{eq:translation-derivative}
  df(q)[\velocity q] = \mtransport p {\frac12(q+r)} \frac12 \mtransport  q p \velocity q = \frac p {\frac12(q+r)} \frac 12 \frac q p \velocity q = \frac12\frac q {\frac12(q+r)} \velocity q = \frac 12 \mtransport q {\frac12(q+r)} \velocity q \ . 
\end{equation}

We now compute the gradient of $q \mapsto \JS q r$ of \cref{eq:JSD}, using the total natural gradient of the KH-divergence of \cref{prop:KL-tot-nat-grad}, the derivative \cref{eq:translation-derivative}, and the duality of parallel transports \cref{eq:transport-duality}:
\vspace{-6pt}
 \begin{multline*}
   \Grad (q \mapsto \JS q r) = \\
   \frac12 \left( - s_q\left(\frac12(q+r)\right) - \frac 12 \etransport {\frac12(q+r)} q \eta_{\frac12(q+r)}(q)\right) + \frac 12 \left(- \frac 12 \etransport {\frac12(q+r)} q \eta_{\frac12(q+r)}(r) \right) = \\
   \frac12 \left( - s_q\left(\frac12(q+r)\right) - \frac 12 \etransport {\frac12(q+r)} q \left(\eta_{\frac12(q+r)}(q) + \eta_{\frac12(q+r)}(r)\right) \right) = \\
   \frac12 \left( - s_q\left(\frac12(q+r)\right) - \frac 12 \etransport {\frac12(q+r)} q \left(\frac{q}{\frac12(q+r)}-1 + \frac{r}{\frac12(q+r)}-1 \right) \right) = \\  - \frac12 s_q\left(\frac12(q+r)\right) \ .  
 \end{multline*}

It is also instructive to use the expression of the Jensen--Shannon divergence in terms of entropies. From \cref{eq:entropy-grad},
\vspace{-6pt}
\begin{multline*}
  \Grad (q \mapsto \JS q r) = \\
  \frac12 \etransport {\frac12(q+r)} q \left( - \log \frac12(q+r) - \entropyof{\frac12(q+r)} \right) - \frac12 \left( - \log q -  \entropyof q \right) = \\ -\frac12 \left(\log \frac12(q+r) - \expectat q {\log \frac12(q+r)} + \log q - \expectat q {\log q} \right) = \\ - \frac12 s_q\left(\frac12(q+r)\right) \ .   
\end{multline*}

\section{Product Sample Space} \label{sec:product-sample-space}

This section uses $\Omega = \Omega_1 \times \Omega_2$ as a factorial sample space. For each $r \in \maxexpat \Omega$, the margins are $r_1 \in \maxexpat {\Omega_1}$ and $r_2 \in \maxexpat{\Omega_2}$. {In the \emph{{mean-field assumption}}, the model equals the tensor product of the margins,
  \begin{equation*}
    \label{eq:product}
    \bar r = r_1 \otimes r_2 \in \maxexpat{\Omega_1} \otimes \maxexpat {\Omega_2} \subset \maxexpat \Omega \ .
  \end{equation*}
  
{The} velocities are, respectively,
  \begin{gather}
    \label{eq:vel-product-1}
    \velocity r(x,y;t) = \derivby t \log r(x,y;t) \ ,  \\  \label{eq:vel-product-2}
   \velocity {\bar r}(x,y;t) = \derivby t \log \bar r(x,y;t) = \velocity r_1(x;t) + \velocity r_2(y;t) \ .
 \end{gather}
 
 Below, we will discuss the optimality of a mean-field approximation.}

\subsection{Product Sample Space: Marginalization}
\label{sec:prod-sample-spac-margin}

The (first) marginalization is
\begin{equation}
    \label{eq:marginalization}
    \Pi_1 \colon \maxexpat{\Omega_1 \times \Omega_2} \ni r \mapsto r_1 \in \maxexpat{\Omega_1} \ , \quad r_1(x) = \sum_b r(x,b) \ .
  \end{equation}

 {{We will compute the bundle derivative of \cref{eq:marginalization} following the scheme of \cref{eq:bundle-derivative}.} }

\begin{Proposition} \label{prop:deriv-of-marginalization}
  The derivative $d\Pi_1$ of the marginalization \cref{eq:marginalization} is 
\begin{equation*}
  \label{eq:deriv-of-marginalization}
  d\Pi_1 \colon \mixbundleat \Omega \ni (r,\velocity r) \mapsto \left(r_1,\condexpat r {\velocity r}{\Pi_1}\right) \in \mixbundleat {\Omega_1}  \ . 
\end{equation*}
\end{Proposition}

\begin{proof}
  In the mixture chart centered at $p_1 \otimes p_2$ and $p_1$, respectively, the expression of the marginalization is
  \begin{multline}
    \label{eq:expression-marginalization}
\eta_{p_1} \circ \Pi_1 \circ \eta_{p_1 \otimes p_2}^{-1}(v) = \frac {\Pi_1 \circ \eta_{p_1 \otimes p_2}^{-1}(v)}{p_1} -1 = \\ \frac {\sum_b (1 + v(\cdot,b)) \cdot p_1 p_2(b)}{p_1} -1 = \sum_b v(\cdot,b) p_2(b) \ .     
\end{multline}

{Note that the expression in \cref{eq:expression-marginalization} is linear. Hence,} the derivative at $v$ in the direction $h$ is $x \mapsto \sum_b h(\cdot,b) p_2(b)$ with $h = \mtransport {r}{p_1 \otimes p_2} \velocity r$ so that the {bundle} derivative is
\begin{multline*}
  \label{eq:marginalization-derivative}
  d\Pi_1(r)[\velocity r] = \mtransport {p_1}{r_1} \sum_b \frac {r(\cdot,b)}{p_1p_2(b)} \velocity q(\cdot,b) p_2(b) = \\
 \frac {p_1} {r_1} \sum_b \frac {r(\cdot,b)}{p_1p_2(b)} \velocity q(\cdot,b) p_2(b) = \sum_b \velocity q(\cdot,b) r_{1|2}(\cdot|b) = \condexpat r {\velocity r} {\Pi_1} \ .
\end{multline*}
\end{proof}

{There is an interesting relation between conditional expectation and mixture transport.} The conditional expectation commutes with the mixture transports,
  \begin{equation*}
    \label{eq:zakai}
  \mtransport {r_1} {q_1} \condexpat r {\mtransport q r v} {\Pi_1} = \condexpat q v {\Pi_1} \ , \quad v \in \mixfiberat q {\Omega} \ .  
\end{equation*}

{It} is a way to express Bayes' theorem for conditional expectations. For all $\phi$,
\begin{multline*}
  \label{eq:zakai-proof}
  \expectat q {\mtransport {r_1} {q_1} \condexpat r {\mtransport q r v} {\Pi_1} \phi(\Pi_1)} = \expectat {q_1} {\frac{r_1}{q_1} \condexpat r {\frac q r v} {\Pi_1} \phi(\Pi_1)} = \\
  \expectat {r_1} {\condexpat r {\frac q r v} {\Pi_1} \phi(\Pi_1)} = 
  \expectat {r} {\condexpat r {\frac q r v} {\Pi_1} \phi(\Pi_1)} =
  \expectat r {\frac qr v \Phi(\Pi_1)} = \expectat q {v \phi(\Pi_1)} \ . 
\end{multline*}

\subsection{Product Sample Space: Mean-Field Approximation}
\label{sec:prod-sample-spac-mean-field}

{The derivative of the joint marginalization
\begin{equation*}
    \label{eq:mean-field}
    \Pi \colon \maxexpat{\Omega_1 \times \Omega_2} \ni r \mapsto r_1\Pi_1(r)  \otimes \Pi_2(r) \in \maxexpat{\Omega_1 \times \Omega_2} \ ,
  \end{equation*}
  follows from the derivative of the marginalization in \cref{eq:vel-product-2}.}
  
  \begin{Proposition} \label{prop:deriv-of-mean-field}
{The derivative $d\Pi$ of the joint marginalization in \emph{{Section}~\ref{sec:prod-sample-spac-mean-field}} is} 
\begin{equation}
  \label{eq:deriv-of-mean-field}
  d\Pi \colon (r,\velocity r) \mapsto \left(r_1 \otimes r_2,\condexpat r {\velocity r}{\Pi_1} + \condexpat r {\velocity r} {\Pi_2}\right) \ . 
\end{equation}
\end{Proposition}
{
\begin{proof}
Compose the partial derivatives with the mapping $(\velocity r_1,\velocity r_2) \mapsto \velocity r_1 \otimes \velocity r_2$    
\end{proof}}

{The decomposition of the velocity {$velocity\,r$} according to \cref{eq:deriv-of-mean-field} provides a better decomposition than $\velocity r_1 + \velocity r_2$ of \cref{eq:vel-product-1,eq:vel-product-2} and provides a definition of the \emph{{mean-field approximation}}.} In the language of ANOVA decomposition of statistical interactions, the derivative part in \cref{eq:deriv-of-mean-field} is the sum of the simple effects of the~velocity,
  \begin{equation*}
    \label{eq:ANOVA-of-velocity}
\velocity q = \velocity q_1 + \velocity q_2 + \velocity q_{12} \ , 
\end{equation*}
where $q_i = \condexpat q {\velocity q}{\Pi_i}$, $i=1,2$, and the last term is the interaction, the $q$-orthogonal residual. See \cite{pistone:2021gsi} for a discussion of the ANOVA decomposition in the context of the statistical~bundle. 

The equation for the total natural gradient of the KL-divergence and the computation of the derivative above provide the natural gradients of the divergence between the joint probability function and the mean-field approximation. In information theory \cite{cover|thomas:2006}, the KL-divergence in \cref{eq:divergence-mean-field-grad-2} is called mutual information. 

\begin{Proposition} \label{prop:divergence-mean-field-grad}
The natural gradients of the divergences of a joint distribution $r$ and its mean-field approximation $\Pi(r)$ are
  \begin{gather}
    \Grad \KL {\Pi(r)} r =  \condexpat {r_1 \otimes r_2} {s_r(r_1 \otimes r_2) } {\Pi_1} + \condexpat {r_1 \otimes r_2} {s_r(r_1 \otimes r_2) } {\Pi_2} - \eta_r(r_1 \otimes r_2)     \label{eq:divergence-mean-field-grad-1} \ .
\\
    \Grad \KL r {\Pi(r)} = - s_{r}(r_1 \otimes r_2) + \condexpat r {\eta_r(r_1 \otimes r_2)}{\Pi_1} + \condexpat r {\eta_r(r_1 \otimes r_2)} {\Pi_2} \ . 
    \label{eq:divergence-mean-field-grad-2}
  \end{gather}
\end{Proposition}

  The conditional terms in \cref{eq:divergence-mean-field-grad-1} depend on the mean-field model; hence, we could express them as a disintegration of $r$. For example,
  \vspace{-6pt}
  \begin{multline*}
    \condexpat {r_1 \otimes r_2} {s_r(r_1 \otimes r_2) } {\Pi_1 = x} = \\
    \sum_y r_2(y) \log \frac {r_1(x)r_2(y)}{r(x,y)} - \sum_{xy} r_1(x) r_2(y) \log \frac {r_1(x)r_2(y)}{r(x,y)} = \\
    \entropyof{r_2,r_{2|1}(\cdot|x)} - \sum_x r_1(x) \entropyof{r_2,r_{2|1}(\cdot|x)} \ ,    \end{multline*}
      where the last term is the conditional entropy $\entropyof{\Pi_2|\Pi_1}$ under $r$.
    
\begin{proof}[Proof of \cref{eq:divergence-mean-field-grad-1}] We find the natural gradient of $r \mapsto \KL {\Pi(r)} r$ by computing  with Equations~\eqref{eq:KL-total-grad} and \eqref{eq:deriv-of-mean-field} the variation along a smooth curve $t \mapsto r(t) \in \maxexpat{\Omega_1 \times \Omega_2}$ such that $r(0) = r$ and $\velocity r(0) = \velocity r$. It holds that 
  \begin{multline*}
  \left. \derivby t \KL {\Pi(r(t))}{r(t)} \right|_{t-0}=   - \scalarat {\Pi(r)} {s_{\Pi(r)}(r)}{d\Pi(r)[\velocity r]} - \scalarat r {\eta_r(\Pi(r))}{\velocity r} = \\
    - \scalarat {r_1 \otimes r_2} {s_{r_1 \otimes r_2}(r)}{\condexpat r {\velocity r}{\Pi_1} + \condexpat r {\velocity r} {\Pi_2}} - \scalarat r {\eta_r(r_1 \otimes r_2)}{\velocity r}
  \end{multline*}
 
 We want to present the first term of the RHS as an inner product at $r$ applied to $\velocity r$. Let us push the inner product from $r_1 \otimes r_2$ to $r$ with \cref{eq:inner-prod-push}. It holds that 
  \begin{multline*}
    \label{eq:12}
-  \scalarat {r_1 \otimes r_2} {s_{r_1 \otimes r_2}(r)}{(\condexpat r {\velocity r}{\Pi_1} + \condexpat r {\velocity r}{\Pi_2})} = \\ - \scalarat {r} {\etransport {r_1 \otimes r_2} r s_{r_1 \otimes r_2}(r)}{\mtransport {r_1 \otimes r_2} r (\condexpat r {\velocity r}{\Pi_1} + \condexpat r {\velocity r}{\Pi_2})} =  \\   
\scalarat {r} {s_r(r_1 \otimes r_2)}{\frac {r_1 \otimes r_2} r (\condexpat r {\velocity r}{\Pi_1} + \condexpat r {\velocity r}{\Pi_2})} = \\
\expectat {r} {s_r(r_1 \otimes r_2)\frac {r_1 \otimes r_2} r (\condexpat r {\velocity r}{\Pi_1} + \condexpat r {\velocity r}{\Pi_2})} =  \\
\expectat {r} {\left(\condexpat r {s_r(r_1 \otimes r_2)\frac {r_1 \otimes r_2} r} {\Pi_1} + \condexpat r {s_r(r_1 \otimes r_2)\frac {r_1 \otimes r_2} r} {\Pi_2} \right) \velocity r} = \\
\scalarat r {\condexpat {r_1 \otimes r_2} {s_r(r_1 \otimes r_2) } {\Pi_1} + \condexpat {r_1 \otimes r_2} {s_r(r_1 \otimes r_2) } {\Pi_2}} {\velocity r} \ .
\end{multline*}

The last equality follows from
\begin{equation*}
  \condexpat r {s_r(r_1 \otimes r_2)\frac {r_1 \otimes r_2} r} {\Pi_i} = \condexpat {r_1 \otimes r_2} {s_r(r_1 \otimes r_2) } {\Pi_i} \ , \quad i=1,2 \ . 
\end{equation*}
\end{proof}

\begin{proof}[Proof of \cref{eq:divergence-mean-field-grad-2}]
\begin{multline*}
\left. \derivby t \KL {r(t)}{\Pi(r(t))} \right|_{t=0} = - \scalarat {r} {s_{r}(r_1 \otimes r_2)} {\velocity r}  - \scalarat {r_1 \otimes r_2} {\eta_{r_1 \otimes r_2}(r)} {d\Pi(r)[\velocity r]} = \\
- \scalarat {r} {s_{r}(r_1 \otimes r_2)} {\velocity r}  - \scalarat {r_1 \otimes r_2} {\eta_{r_1 \otimes r_2}(r)} {\condexpat r {\velocity r} {\Pi_1} + \condexpat r {\velocity r} {\Pi_2}} \end{multline*}
and compute the second term as
\begin{multline*}
  - \scalarat {r_1 \otimes r_2} {\eta_{r_1 \otimes r_2}(r)} {\condexpat r {\velocity r} {\Pi_1} + \condexpat r {\velocity r} {\Pi_2}} = \\
  - \scalarat {r} {\mtransport {r_1 \otimes r_2} r \eta_{r_1 \otimes r_2}(r)} {\etransport {r_1 \otimes r_2} r \left(\condexpat r {\velocity r} {\Pi_1} + \condexpat r {\velocity r} {\Pi_2}\right)} = \\
\scalarat {r} {\eta_r(r_1 \otimes r_2)} {\condexpat r {\velocity r} {\Pi_1} + \condexpat r {\velocity r} {\Pi_2}} = \\
\expectat r {\eta_r(r_1 \otimes r_2) \left(\condexpat r {\velocity r} {\Pi_1} + \condexpat r {\velocity r} {\Pi_2}\right)} = \\
\expectat r {\left(\condexpat r {\eta_r(r_1 \otimes r_2)}{\Pi_1} + \condexpat r {\eta_r(r_1 \otimes r_2)} {\Pi_2}\right) \velocity r}
= \\
\scalarat r {\condexpat r {\eta_r(r_1 \otimes r_2)}{\Pi_1} + \condexpat r {\eta_r(r_1 \otimes r_2)} {\Pi_2}}{\velocity r} \ .
\end{multline*}
\Cref{eq:divergence-mean-field-grad-2} follows.
\end{proof}

 \subsection{Product Sample Space: Kantorovich and Schr\ "odinger}

If $\Pi$ denotes the {joint marginalization}, the set of \emph{transport plans} with margins $q_1$ and $q_2$~is 
\begin{equation*}
    \Gamma(q_1,q_2) = \Phi^{-1}(q_1 \otimes q_2) = \setof{q}{\Pi(q) = q_1 \otimes q_2} \ .
\end{equation*}

{{Here}, we deal with a classical topic with considerable literature}. We mention only the monograph of ref. \cite{peyre|cuturi:2019} and, from the Information Geometry perspective, ref. \cite{amari|karakida|oizumi:2018INGE,pistone:2021gsi}.

Let us consider first the Kantorovich problem. Given the cost function (i.e., potential~function) 
\begin{equation*}
U \colon \Omega \to \reals \ ,    
\end{equation*}
and a curve $t \mapsto q(t) \in \gamma(q_1,q_2)$, we want to minimize the cost
\begin{equation*}
  S(t) = \expectat {q(t)} U \ .
\end{equation*}

As $\expectat {q(t)}{ \phi(\Pi_i)} = \expectat {q_i} \phi$ for all $\phi$,
\begin{equation*}
0 = \derivby t \scalarat {q(t)}{U - \expectat {q(t)} U}{\velocity q(t)} \ ,  
\end{equation*}
so that $\condexpat {q(t)} {\velocity q(t) }{ \Pi_i} = 0$. The velocity of a curve in the transport plans is an interaction. {Now,} the derivative of the cost is
\begin{equation} \label{eq:deriv-mean}
\derivby t S(q(t)) = \scalarat {q(t)} {U - \expectat {q(t)} U}{\velocity q(t)} \ .
\end{equation}

{From the interaction property of $\velocity q(t)$, it follows that if the ANOVA decomposition}
\begin{equation*}
  U = \expectat {q} U + (u_1(\Pi_1;q) + u_2(\Pi_2;q)) + u_{12}(\Pi_1,\Pi_2;q)
\end{equation*}
holds, then
\begin{equation*}
  \derivby t S(q(t)) = \scalarat {q(t)} {u_{12}(\Pi_1,\Pi_2;q(t))}{\velocity q(t)} \ .
\end{equation*}

The Scr\" odinger problem is similar. Given the cost function (i.e., potential function) 
\begin{equation*}
U \colon \Omega \to \reals \ ,    
\end{equation*}
consider the exponential perturbation of the mean-field probability function 
\begin{equation}\label{eq:epsilon-temperature}
    \exp\left(-\frac U \epsilon - \phi(\epsilon)\right) \cdot q_1 \otimes q_2  \ .
\end{equation}

{The} parameter $\epsilon > 0$ is called \emph{{temperature}}, and the normalizing constant is
\begin{equation*}
    \psi(\epsilon) = \log \expectat {q_1 \otimes q_2} {\euler^{-U/\epsilon}} \ .
\end{equation*}

The KL-divergence of $q$ relative to the perturbed probability function of \cref{eq:epsilon-temperature}~is
\begin{multline*} \label{eq:S-problem}
    S_\epsilon(q) = \KL q {\euler^{-U/\epsilon - \phi(\epsilon)} \cdot (q_1 \otimes q_2) } = \\
    \expectat q {\log \frac q {\log q_1 \otimes q_2}} + \epsilon^{-1} \expectat {q} U + \psi(\epsilon) = \\
     \epsilon^{-1}\left(S(q) + \epsilon \KL q {q_1 \otimes q_2} + \epsilon \phi(\epsilon)\right) \ .
 \end{multline*}

The gradient of $q \mapsto S_\epsilon (q)$ is, from Equations~\eqref{eq:deriv-mean} and \eqref{eq:divergence-mean-field-grad-2},
\begin{multline*}
  \Grad S_\epsilon (q) = \\
  \epsilon^{-1}\left(U - \expectat q U - s_{q}(q_1 \otimes q_2) + \condexpat q {\eta_q(q_1 \otimes q_2)}{\Pi_1} + \condexpat q {\eta_q(q_1 \otimes q_2)} {\Pi_2} \right) \ .
\end{multline*}

Only the interaction part is relevant in the constrained problem $q \in \Gamma(q_1,q_2)$, and the interaction kills the two conditional expectations, which leaves
\begin{equation*}
    \left( U - \expectat q U \right))_{12;q} - \left(s_q(q_1 \otimes q_2)\right)_{12;q} \ .
    \end{equation*}

  We refer to \cite{pistone:2021gsi} for a method to compute the interaction part of a random variable. 

\subsection{Product Sample Space: Conditional Probability Function}

When the sample space is a product, $\Omega = \Omega_1 \times \Omega _2$, we can represent each probability function in the maximal exponential model via conditioning on one margin,
\begin{align}
  \maxexpat{\Omega_1 \times \Omega_2} &= \setof{q = q_{1|2} \cdot q_2}{q_{1|2}(\cdot|y) \in \maxexpat {\Omega_1}, y \in \Omega_2, q_2 \in \maxexpat{\Omega_2}} \\ 
  &= \setof{q = q_{2|1} \cdot q_1}{q_{2|1}(\cdot|x) \in \maxexpat {\Omega_2}, x \in \Omega_1, q_1 \in \maxexpat{\Omega_1}} \ .
\end{align}

The two representations are
\begin{equation} \label{eq:bayes}
   \maxexpat {\Omega_1}^{\otimes\Omega_2} \times \maxexpat {\Omega_2} \leftrightarrow \maxexpat{\Omega_1 \times \Omega_2} \leftrightarrow \maxexpat{\Omega_1} \times \maxexpat{\Omega_2}^{\otimes \Omega_1} \ .
\end{equation}

Following the approach of \cite{goodfellow|pouget-abadie|mirza|xu|warde-farley|ozair|courville|bengio:2014}, (\cite{amari:2016}, Ch.~11), we look at the transition mapping
\begin{equation*}
  \Omega_2 \ni y \mapsto q_{1|2}(\cdot|y) \in \maxexpat {\Omega_1}
\end{equation*}
as a family of probability functions representing alternative probability models. The other transition mapping
\begin{equation*} \label{eq:discriminator}
\Omega_1 \ni x \mapsto q_{2|1}(\cdot|x) \in \maxexpat{\Omega_2}
\end{equation*}
is the \emph{{discriminator}}, that is, $q_{2|1}(y|x)$ is the probability that the sample $x$ comes from $q_{1|2}(\cdot|y)$.

The right-to-left second mapping in \cref{eq:bayes}, 
\begin{equation}
  \label{eq:bayes-mapping}
  B \colon \maxexpat{\Omega_1} \times \maxexpat{\Omega_2}^{\otimes \Omega_1} \ni \left(q_1, q_{2|1}(\cdot|x), x \in \Omega_1 \right)
    \mapsto q = q_{2|1} \cdot q_1 \in \maxexpat{\Omega_1 \times \Omega_2} \ ,
  \end{equation}
maps the vector of the 1-margin and the set of alternative probability functions to the joint probability function. The kinematics of \cref{eq:bayes-mapping}, that is, the computation of velocities,~is
\begin{equation*}
  \velocity q(x,y;t) = \velocity q_1(x;t) + \velocity q_{2|1}(y|x;t) \ , 
\end{equation*}

{Hence}, the total derivative of $B$ is
\begin{equation*}
  \label{eq:bayes-mapping-deriv}
  dB(q_1,q_{2|1}(\cdot|x),x\in\Omega_1)[\velocity q_1,\velocity q_{2|1}(\cdot|x),x\in\Omega_1] \colon (x,y) \mapsto \velocity q_1(x) + \velocity q_{2|1}(y|x) \ .
\end{equation*}

The transposed total derivative is defined by
\begin{equation*}
  \scalarat {B(q_1,q_{2|1})} {v} {dB(q_1,q_{2|1})[\velocity q_1,\velocity q_{2|1}]} =
  \scalarat {(q_1,q_{2|1})} {dB(q_1,q_{2|1})^*[v]}{(\velocity q_1,\velocity q_{2|1})} \ ,
\end{equation*}
that is,
\begin{multline*}
  \sum_{x,y} v(x,y) \left(\velocity q_1(x) + \velocity q_{2|1}(y|x)\right)q_1(x) q_{2|1}(y|x) = \\
  \sum_{x,y} v(x,y) \velocity q_1(x) q_1(x) q_{2|1}(y|x) + \sum_{x,y} v(x,y) \velocity q_{2|1}(y|x) q_1(x) q_{2|1}(y|x) = \\
  \sum_x \left(\sum_y v(x,y) q_{2|1}(y|x)\right) \velocity q_1(x) q_1(x) + \sum_x \sum_y q_1(x) v(x,y)  \velocity q_{2|1}(y|x) q_{2|1}(y|x) = \\
  \scalarat {q_1} {\condexpat q v {\Pi_1}}{\velocity q_1} + \sum_x q_1(x) \sum_y v(x,y) \velocity q_{2|1}(y|x) q_{2|1}(y|x) = \\
  \scalarat {q_1} {\condexpat q v {\Pi_1}}{\velocity q_1} + \sum_x \scalarat { q_{2|1}(\cdot|x)} {q_1(x) \left(v(x,\cdot) - \expectat {q_{2|1}(\cdot|x)} {v(x,\cdot)} \right)} {\velocity q_{2|1}(y|x)} \ .
\end{multline*}

In conclusion, the transposed total derivative is
\begin{equation*}
  dB(q_1,q_{2|1})^*[v] \colon (x,y) \mapsto \left(\condexpat q v {\Pi_1},q_1(x) \left(v(x,\cdot) - \expectat {q_{2|1}(\cdot|x)} {v(x,\cdot)} \right),x \in \Omega_1\right) \ .
\end{equation*}

 It is interesting to derive $dB$ in the mixture atlas. The mixture expression of $B$ with respect to $(p_1,p_2^{\otimes \Omega_1})$ and $p_1 \otimes p_2$ is
  \begin{multline*}
     B_{p_1,p_2} \colon (v_1,v_{2|1}(\cdot|x),x \in \Omega_1) \mapsto \\
    (x,y) \mapsto (1+v_1(x)) p_1(x) (1+v_{2|1}(y|x)) p_2(y) \mapsto \\
    (1 + v_1)(1+v_{2|1}) - 1
  \end{multline*}
  and the total derivative in the directions $h_1$, $h_{2|1}$ is
  \begin{multline*}
    dB_{p_1,p_2}(v_1,v_{2|1}(\cdot|z),z \in \Omega_1)[h_1,h_{2|1}(\cdot|z),z \in \Omega_1] = \\
    (1+v_{2|1})h_1 + \sum_z (1+v_1(z))h_{2|1}(\cdot|z) \ .
  \end{multline*}

  {{The} push-back of the total derivative expression to the statistical bundles uses the equations}
  \begin{equation*}
    1+v_1 = q_1/p_1, 1+v_{2|1} = q_{2|1}/p_2, h_1=\mtransport {q_1} {p_1} \velocity q_1, h_{2|1}(\cdot | z) = \mtransport {q_{2|1}(\cdot|z)}{p_2} \velocity q_{2|1}(\cdot|z)
  \end{equation*}
  to obtain
  \begin{multline*}
    dB(q_1,q_{2|1}(\cdot|z),z \in \Omega_1)[\velocity q_1,\velocity q_{2|1}(\cdot|z),z \in \Omega_1] = \\
    \mtransport {p_1 \otimes p_2}{q_{2|1} \cdot q_1} \left(\frac{q_{2|1}}{p_2} \mtransport {q_1} {p_1} \velocity q_1 + \sum_z \frac{q_1(z)}{p_1(z)} \mtransport {q_{2|1}(\cdot|z)}{p_2} \velocity q_{2|1}(\cdot|z) \right) = \\
    \frac {p_1 \otimes p_2}{q_{2|1} \cdot q_1} \left(\frac{q_{2|1}}{p_2} \frac {q_1} {p_1} \velocity q_1 + \sum_z \frac{q_1(z)}{p_1(z)} \frac {q_{2|1}(\cdot|z)}{p_2} \velocity q_{2|1}(\cdot|z) \right)
  \end{multline*}
  
  In our affine language, we repeat computations in \cite{goodfellow|pouget-abadie|mirza|xu|warde-farley|ozair|courville|bengio:2014}. We especially derive the  natural gradient of a composite function by the equation
\begin{equation*}
  \Grad \Phi \circ B = dB^* \Grad \Phi (B) \ .
\end{equation*}

For a given target $p \in \maxexpat{\Omega_1 \times \Omega_2}$, we express the KL-divergence as a function of $B$ in \cref{eq:bayes-mapping},
\vspace{-6pt}
 \begin{multline*}
  K_q \colon (r_1,r_{2|1}(\cdot|x),x \in \Omega_1) \mapsto \KL p {B(q_1,q_{2|1}(\cdot|x),x \in \Omega_1)} = \\ \sum_{x,y} p(x,y) \log \frac {p(x,y)}{q_1(x)q_{2|1}(y|x)} \ .
\end{multline*}

We have
\begin{equation*}
  \Grad K_p(B(q_1,q_{2|1})) = - \eta_{q_1 \cdot q_{2|1}}(p) = 1 - \frac p {q_1 \cdot q_{2|1}}  \ .
\end{equation*}

The first component of $\Grad K_p \circ B$ is 
\begin{multline}
  \label{eq:GAN-1}
  [dB(q_1,q_{2|1})^*\Grad K_p (B(q_1,q_{2|1}))]_1 \colon x \mapsto \\ \sum_z \left(1 - \frac {p(x,z)} {q_1(x) q_{2|1}(z|x)}\right) q_{2|1}(z|x) = \sum_z \frac{q_1(x) q_{2|1}(z|x) - p(x,z)}{q_1(x)} = 1 - \frac {p_1(x)}{q_1(x)} \ , 
  \end{multline}
  so that $[\Grad K_p \circ B(q_1,q_{2|1})]_1 = - \eta_{q_1}(p_1)$.
  The $x$-component is
  \begin{multline}
    \label{eq:GAN-x}
    [dB(q_1,q_{2|1})^*\Grad K_p (B(q_1,q_{2|1}))]_x \colon y \mapsto \\
    q_1(x) \left(1 - \frac{p(x,y)}{q_1(x)q_{2|1}(y|x)}\right) - q_1(x) \sum_y  \left(1 - \frac{p(x,y)}{q_1(x)q_{2|1}(y|x)}\right) q_{2|1}(y|x) = \\
    - \frac {p(x,y)}{q_{2|1}(y|x)} + p_1(x) \ ,
  \end{multline}
  so that $ [dB(q_1,q_{2|1})^*\Grad K_p (B(q_1,q_{2|1}))]_x  = - p_1(x) \eta _{q_{2|1}(\cdot|x)}(p_{2|1}(\cdot|x))$. 
  
We now assume a target probability function $g \in \maxexpat{\Omega_1}$ and consider the probability function on the product sample space where all the model probability functions equal $g$, and the discriminator is uniform, say $p = \frac 1m g$, $m = \# \Omega_2$.  

In this case, \cref{eq:GAN-1} becomes
\begin{equation*}
 \Grad K_p \circ B(q_1,q_{2|1})]_1 = - \eta_{q_1}(g) 
\end{equation*}
and \cref{eq:GAN-x} becomes
\begin{equation*}
 \Grad K_p \circ B(q_1,q_{2|1})]_x = - q_1(x)\left(1 - \frac{g(x)/m}{q_1(x)q_{2|1}(y|x)}\right)
\end{equation*}

{
\subsection{Variational Bayes}
\label{sec:variational-bayes}

We revisit and develop some computations of (\cite{kingma|welling:2022autoencodingvariationalbayes}, \S~2.2). We keep the same notation as above so that Bayes' formula is
\begin{equation*}
  q_{2|1}(y|x) = \frac {q_{1|2}(x|y) q_2(y)} {q_1(x)} = \frac {q_{1|2}(x|y) q_2(y)} {\sum_y q_{1|2}(x|y) q_2(y)} \ ,
\end{equation*}
where $x$ is a sample value and $y$ is a latent variable value.

For a fixed  $x \in \Omega_1$, we look for a $r$ in some model $\mathcal M \subset  \maxexpat {\Omega_2}$ in order to approximate the conditional $q_{2|1}(\cdot|x)$. If $q \mapsto \mathcal L(r,x)$ satisfies 
\begin{equation*}
  \log q_1(x) = \KL {r} {q_{2|1}(\cdot|x)} + \mathcal L(r,x) \ ,
\end{equation*}
then
\begin{multline*}
  \mathcal L(r,x) = \log q_1(x) - \sum_y r(y) \log \frac {r(y)} {q_{2|1}(y|x)} = \\
  \sum_y r(y) \left(\log q_1(x) - \log \frac {r(y)} {q_{2|1}(y|x)}\right) = \sum_y r(y) \log \frac {q_{2|1}(y|x) q_1(x)}{r(y)} = \\
  \sum_y r(y) \log \frac {q_{12}(x,y)}{r(y)} = \sum _y r(y) \log \frac {q_2(y)} {r(y)} + \sum_y r(y) \log \frac {q_{12}(x,y)}{q_2(y)} = \\ - \KL {r} {q_2} + \sum_y r(y) \log q_{1|2}(x|y) \ .  
\end{multline*}

The so-called \emph{{variational lower bound}} follows from $\KL {r} {q_{2|1}(\cdot|x)} \geq 0$,
\begin{multline*}
  \log q_1(x) = \KL {r} {q_{2|1}(\cdot|x)} - \KL {r} {q_2} + \expectat r {\log q_{1|2}(x|\cdot)} \ge \\
 - \KL {r} {q_2}  + \expectat r {\log q_{1|2}(x|\cdot)} = \mathcal L(r;x) \ .
\end{multline*}
for all $r \in \mathcal M$. The bound is \emph{{exact}} because $r(y) = q_{2|1}(y|x)$ for all $y$ if, and only if, $\KL r {q_{2|1}(\cdot|x)} = 0$ 
The lower-bound variation along a curve $t \mapsto r(t) \in \mathcal M$ is
\begin{multline} \label{eq:derviv-of-L}
\derivby t \mathcal L(r(t);x) = \\ \scalarat {r(t)} {\velocity r(t)}{s_{r(t)}(q_2)} + \scalarat {r(t)} {\velocity r(t)} {\log q_{1|2}(x|\cdot) - \expectat {r(t)} {\log q_{1|2}(x|\cdot)}} = \\ \scalarat{r(t)}{\velocity r(t)}{\log \frac {q_2}{r(t)} + \log q_{1|2}(x|\cdot) - \expectat {r(t)} {\log \frac {q_2}{r} + \log q_{1|2}(x|\cdot)}} = \\
\scalarat{r(t)}{\velocity r(t)}{\log \frac {q_{12}(x,\cdot)}{r(t)} - \expectat {r(t)} {\log \frac {q_{12}(x,\cdot)}{r(t)}}} \ .
\end{multline}

If the model $\mathcal M$ is an exponential tilting of the margin $q_2$,
\begin{equation*}
    r = e_{q_2}(\theta \cdot u) = \euler^{\theta \cdot u - \psi(\theta)} \cdot q_2 \ ,
\end{equation*}
where $\theta \in \Theta \subset \reals^d$ is a vector parameter, and $u$ is the vector of sufficient statistics of the exponential family with $\expectat {q_2} u = 0$, then the velocity  in \cref{eq:derviv-of-L} becomes
\begin{equation*}\velocity r(t) = \dot \theta(t) \cdot (u - \nabla \psi(\theta(t)) = \dot \theta(t) \cdot (u - \expectat {r(t)} u) = \dot \theta(t) \cdot \etransport {q_2} {r(t)} u \ ,
\end{equation*}
and the gradient in \cref{eq:derviv-of-L} becomes
\begin{multline*}
\log \frac {q_{12}(x,\cdot)}{r(t)} - \expectat {r(t)} {\log \frac {q_{12}(x,\cdot)}{r(t)}} = \\
\log \frac {q_{12}(x,\cdot)} {q_2} - (\theta(t) \cdot u - \psi(u)) - \expectat{r(t)}{\log \frac {q_{12}(x,\cdot)} {q_2} - (\theta(t) \cdot u - \psi(u))} = \\
\log q_{1|2}(x|\cdot) - \theta(t) \cdot u  - \expectat{r(t)}{\log q_{1|2}(x|\cdot) - \theta(t) \cdot u} = \\
\etransport {q_2} {r(t)} \left(\log q_{1|2}(x|\cdot) - \expectat{q_2}{\log q_{1|2}(x|\cdot)}- \theta(t) \cdot u\right)  \ .
\end{multline*}

{Using} repeatedly \cref{eq:Kp-2-deriv}, we find the derivative of the lower bound in \mbox{\cref{eq:derviv-of-L}},
\vspace{-12pt}
\begin{multline*} 
\derivby t \mathcal L(r(t);x)  = \\
\scalarat{r(t)}{\dot \theta(t) \cdot  \etransport {q_2} {r(t)} u}{\etransport {q_2} {r(t)} \left(\log q_{1|2}(x|\cdot) - \expectat{q_2}{\log q_{1|2}(x|\cdot)} - \theta(t) \cdot u\right)} = \\ 
\sum_{i=1}^d \dot \theta_i(t) \left(\covat{r(t)}{u_i}{\log q_{1|2}(x|\cdot)} - \sum_{j=1}^d \theta_j(t) \covat{r(t)}{u_i}{u_j}\right) = \\
\dot \theta(t) \cdot \left( \covat{r(t)}{u}{\log q_{1|2}(x|\cdot)} - \Hessian \psi(\theta(t)) \theta(t)\right)
\ .
\end{multline*}

In conclusion, the gradient flow equation for the maximization of the lower bound under the model $\mathcal M$ is
\begin{equation} \label{eq:gradient-flow-VB}
    \dot \theta = - \Hessian \psi(\theta(t)) \theta(t) + \covat{e_{q_2}(\theta(t) \cdot u)}{u}{\log q_{1|2}(x|\cdot)} \ .
\end{equation}

As a sanity check, assume that the model is exact for the given $x$,
\begin{equation*}
    q_{2|1}(y|x) = \euler^{\bar \theta \cdot u(y) - \psi(\bar \theta)}  \cdot q_2(y) \ .
\end{equation*}

{Hence,}
\begin{equation*}
    \log q_{1|2}(x|y) = \log \euler^{\bar \theta \cdot u(y) - \psi(\bar \theta)} q_2(y) \frac {q_1(x)} {q_2(y)} =
    \bar \theta \cdot u = \psi(\bar \theta) + \log q_1(x) \ ,
\end{equation*}
where the last two terms do not depend on $y$ so that
\begin{equation*}
    \covat{e_{q_2}(\theta \cdot u}{u}{\bar \theta \cdot u} = \Hessian \psi(\theta) \bar \theta
\end{equation*}
and \cref{eq:gradient-flow-VB} becomes
\begin{equation*}
    \dot \theta = - \Hessian \psi(\theta(t)) (\theta(t) - \bar \theta)\ .
\end{equation*}

The solution of the gradient flow Equation \eqref{eq:gradient-flow-VB} requires the ability to compute the covariance for the current model distribution. We do not discuss the numerical and simulation issues related to the implementation here.}

\section{Discussion}
\label{sec:conclusions}

In this paper, we have shown how the dually affine formalism for the open probability simplex provides a system of affine charts in which the statistical notion of Fisher's score becomes the moving-chart affine velocity, and the natural gradient becomes the gradient. The construction applies to standard computations of interest for statistical machine learning. In particular, we have discussed the neat form of the total gradient of the KL-divergence and its applications in a factorial sample space, such as mean-field approximation and Bayes computations.

This approach is helpful because the unique features of the Fisherian approach to statistics, such as Fisher's score, maximum likelihood, and Fisher's information, are formalized as an affine calculus so that all the statistical tools are available in this more extensive theory. Moreover, this setting potentially unifies the formalisms of Statistics, Optimal Transport, and Statistical Physics, examples being the affine modeling of Optimal Transport~\cite{ay:2024InformarionGeometry} and the second-order methods of optimization \cite{chirco|malago|pistone:2022}. 

{We have not considered the implementation of the formal gradient flow equation as a practical learning algorithm. Such further development is currently outside the scope of this piece of research. It would require the numerical analysis of the continuous equation and the search for sampling versions of the expectation operators. We hope this note will prompt further research. On the abstract side, topics worth studying seem to be the cases of continuous sample space as in \cite{pistone:2018-igaia-iv}, or Gaussian models as in \cite{malago|montrucchio|pistone:2018}.}
 
\bibliographystyle{amsplain}

\providecommand{\bysame}{\leavevmode\hbox to3em{\hrulefill}\thinspace}
\providecommand{\MR}{\relax\ifhmode\unskip\space\fi MR }
\providecommand{\MRhref}[2]{%
  \href{http://www.ams.org/mathscinet-getitem?mr=#1}{#2}
}
\providecommand{\href}[2]{#2}

\end{document}